\documentclass[12pt]{amsart}

\usepackage{amsmath, amssymb, amscd, verbatim, xspace,amsthm}
\usepackage{latexsym, epsfig, color}
\usepackage[hidelinks]{hyperref}

\usepackage[OT2,T1]{fontenc}
\DeclareSymbolFont{cyrletters}{OT2}{wncyr}{m}{n}
\DeclareMathSymbol{\Sha}{\mathalpha}{cyrletters}{"58}

\newcommand{\ba}{\begin{align*}}
\newcommand{\ea}{\end{align*}}

\newcommand{\C}{\ensuremath{{\mathbb{C}}}}
\newcommand{\Z}{\ensuremath{{\mathbb{Z}}}\xspace}
\renewcommand{\P}{\ensuremath{{\mathbb{P}}}}

\newcommand{\R}{\ensuremath{{\mathbb{R}}}}

\newcommand{\E}{\ensuremath{{\mathbb{E}}}}

\newcommand{\ra}{\rightarrow}

\newcommand\Hom{\operatorname{Hom}}

\newcommand\Aut{\operatorname{Aut}}

\newcommand\Sur{\operatorname{Sur}}

\newcommand\tensor{\otimes}
\newcommand\isom{\simeq}

\newcommand\tesnor{\otimes}

\newcommand\cok{\operatorname{cok}}

\newcommand\bq{\begin{equation}}
\newcommand\eq{\end{equation}}

\newtheorem{proposition}{Proposition}[section]
\newtheorem{theorem}[proposition]{Theorem}
\newtheorem{corollary}[proposition]{Corollary}

\newtheorem{lemma}[proposition]{Lemma}

\newtheorem{conjecture}[proposition]{Conjecture}

\theoremstyle{remark}

\usepackage{fullpage,pdfsync}

\newtheorem{nts}{Note to self}

\title{Cokernels of adjacency matrices of random $r$-regular graphs}

\author{Hoi H. Nguyen}
\address{Department of Mathematics\\ The Ohio State University \\ 231 W 18th Ave \\ Columbus, OH 43210 USA}
\email{nguyen.1261@math.osu.edu}

\author{Melanie Matchett Wood}
\address{Department of Mathematics\\
University of Wisconsin-Madison \\ 480 Lincoln Drive \\
Madison, WI 53705 USA}  
\email{mmwood@math.wisc.edu}

\begin{document}
\begin{abstract}
We study the distribution of the cokernels of adjacency matrices (the Smith groups) of certain models of random $r$-regular graphs and directed graphs, using recent mixing results of M\'esz\'aros.  We explain how convergence of such distributions to a limiting probability distribution implies asymptotic nonsingularity of the matrices, giving another perspective on recent results of Huang and M\'esz\'aros on asymptotic nonsingularity of adjacency matrices of random regular directed and undirected graphs, respectively.
We also remark on the new distributions on finite abelian groups that arise, in particular in the $p$-group aspect when $p\mid r$.
\end{abstract}

\maketitle

\section{Introduction}

The singularity problem in combinatorial random matrix theory states that if a square matrix $A_n$ of size $n$ is ``sufficiently
random'', then $A_n$ is non-singular asymptotically almost surely as
$n$ tends to infinity, in other words $p_n$, the probability of $A_n$
being singular, tends to 0. This problem has a rich history, for which
we now mention briefly. In the early 60s Koml\'os \cite{Komlos67}
showed that if the entries of $A_n$ take values $\{0,1\}$
independently with probability 1/2 then $p_n=O(n^{-1/2})$. This bound
was significantly improved by Kahn, Koml\'os and Szemer\'edi
\cite{KKSz95} to $p_n \le 0.999^n$ for random $\pm 1$ matrices, by Tao
and Vu \cite{TV07} to $p_n \le (\frac{3}{4}+o(1))^n$, by Rudelson and
Vershynin \cite{RV08}, and by Bourgain, Vu and Wood \cite{BVW10}
where it was shown that $p_n= (\frac{1}{\sqrt{2}}+o(1))^n$ for random
$\pm 1$ matrices. The methods of these results also give exponential
bounds for other more general iid ensembles. Since then, there have been
subsequent papers addressing the sparse cases, such as those by Wood
\cite{PWood12}, and by Basak and Rudelson \cite{AR17} where in the
later result the entries of $A_n$ can take values 0 with probability as large as
$1-O(\frac{\log n}{n})$. We refer the reader to \cite{TV10, PWood12, AR17} and
the references therein to various extension and application of the
singularity problem for the iid models. 

This singularity problem views the $A_n$ as matrices over $\R$, but since the entries are integers they could also be viewed as elements of the field $\Z/p\Z$ for any prime $p$.  A matrix is singular mod $p$ exactly when its determinant is $0$ mod $p$, and so heuristically, one expects this to happen about $1/p$ of the time instead of $0\%$ of the time.  From this point of view, it is natural to ask more refined questions, including the limiting distribution of ranks of the matrices mod $p$, or the distribution of the cokernels $\Z^n/A_n(\Z^n)$ of the matrices, i.e. the failure of the matrices to be surjective as maps $A_n: \Z^n\ra \Z^n$.   The second author  determined the limiting distribution for the cokernels of random matrices, including the $A_n$ considered above \cite{Wood2015a}.  For square matrices,  these distributions have non-zero limiting behavior if one considers the Sylow $p$-subgroups of the cokernels for only finitely many primes at once.  However, for non-square matrices, recent work of the authors \cite{Nguyen2018a} finds the distribution of the entire cokernel.  The distributions of cokernels that arise are the Cohen-Lenstra distributions on finite abelian groups \cite{Cohen1984}.

In another direction, there
have been results studying the singularity problem for matrices with various dependency conditions on the entries. For
instance in \cite{Nguyen13} the first author studied random (dense) matrices of given
row sums, or in  \cite{AChW16} Adamczak, Chafai and Wolff studied
random matrices
with exchangeable entries. 
Cook \cite{Cook2017} studied the singularity of adjacency matrices of random $r$-regular digraphs where he showed that $p_n =r^{-\Omega(1)}$
as long as $\min(r,n-r) \geq C \log^2 n$ for some absolute constant
$C$. A similar result was also established by Basak, Cook and
Zeitouni \cite{BCZ18} for sum of $r$ random permutation
matrices as long as $r \ge \log^{12-o(1)} n$. While these results are
highly non-trivial, the random matrices are still relatively
dense and so one might still believe that the matrices are still non-singular
with high probability. However, it has been conjectured that the
phenomenon continue to hold for extremely sparse matrices with some row and
column constraints. Let us mention here one such conjecture. For a positive integer $r$, let
$A_{n,r}$ be a uniformly distributed matrix on the set of all
$\{0,1\}$ square
matrices of size $n$ with $r$ ones in every row and column (i.e. $A_{n,r}$ is the
adjacency matrix of a random $r$-regular directed graph on $n$
vertices where loops are allowed.) Motivated by
Conjecture \ref{conj:sym} (to be mentioned below), the following has
been asked by Cook  \cite[Conjecture 2]{Cook2017}.
\footnote{It was originally stated as $p_n=O(n^{-c})$ for some absolute
    positive constant $c$.}  

\begin{conjecture}\label{conj:non-sym} For any $3\le r\le n-3$ we have
  $p_n=o(1)$. 
\end{conjecture}
For this model, the recent work by Litvak, Lytova, Tikhomirov,
Tomczak-Jaegermann and Youssef in \cite{Litvak2017} shows that $p_n \leq \frac{C\log^3r}{\sqrt{r}}$ as long as $C\leq r \leq
cn/\ln^2 n$ for some constants $c,C$. As a consequence, this bound implies that $p_n \to
\infty$ if $r \to \infty$. By a
more involved study of the structure of the eigenvectors of $A_{n,r}$, it has been
shown by the same goup of authors in \cite{Litvak2018a} that
asymptotically almost surely the rank of $A_{n,r}$ is at least $n-1$
as long as  $r>C$ for sufficiently large constant $C$. Finally, very
recently Huang \cite{Huang2018} confirmed Conjecture \ref{conj:non-sym}
for this regime of $r$ and proved the following. 

\begin{theorem}\label{T:D} Conjecture \ref{conj:non-sym} is true for fixed $r$. 
\end{theorem}

In this note we will give another proof of Theorem~\ref{T:D} (see Theorem~\ref{T:main2}), via determination of the distribution of the cokernels of related matrices. Specifically we show the following.

\begin{theorem}\label{T:sandpile2}
For an integer $r\geq 3$, and an integer $n,$ let $D_{n}$ be the sum of $r$ independent uniform random $n\times n$ permutation matrices.
For  a finite set $P$ of  primes not dividing $r$, and a finite abelian group $V$ such that $|V|$ is a product of powers of primes in $P$, we have
\begin{equation}\label{E:probs2}
\lim_{\substack{n\ra\infty}} \P( \cok D_n\tensor\prod_{p\in P} \Z_p \isom V)=
 \frac{1}{|\Aut(V)|}\prod_{p\in P} \prod_{k\geq 1}(1-p^{-k}).
\end{equation}

\end{theorem}

Our method to prove Theorems~\ref{T:sandpile2} (and the symmetric version, Theorem~\ref{T:sandpile} below) uses recent powerful mixing results of M\'esz\'aros \cite{Meszaros2018} to find the moments of this distribution and results of the second author  \cite{Wood2017} on the moment problem for random finite abelian groups.  We note that M\'esz\'aros \cite{Meszaros2018} proved his mixing results as a step towards determining the distribution of the cokernels of the Laplacians of the random graphs from Theorems~\ref{T:sandpile2} and Theorem~\ref{T:sandpile}, and so it is not surprising they are also useful for the closely related adjacency matrices.  The distribution in Theorem~\ref{T:sandpile2} is the Cohen-Lenstra distribution, as appears in the  case of independent entries \cite{Wood2015a} and in the Laplacian version of the underlying graphs \cite{Meszaros2018}.  However, unlike in these previous cases, the entire cokernels of $D_n$ are not distributed in the Cohen-Lenstra distribution, but rather only the prime-to-$r$ part of $\cok D_n$ follows the 
Cohen-Lenstra distribution.  In Section~\ref{S:newd}, we find the limiting moments for the distribution of the entire cokernel and remark on some new and interesting behavior of these distributions at primes dividing $r$.  

Now we discuss symmetric random matrices.
 Answering a question of Weiss, for symmetric random matrices $A_n$,  Costello, Tao and Vu
\cite{Costello06} showed in 2006 that if the upper diagonal entries of $A_n$
take values $\{0,1\}$ independently with probability 1/2 then
$p_n=O(n^{-\Omega(1)})$. This bound has been improved by the first author
in \cite{Nguyen2012} to $p_n =O(n^{-\omega(1)})$ and then to
$p_n=O(\exp(-n^c))$ in parallel by Vershynin \cite{Vershynin14}. For
sparse matrices,  it was shown by Costello and Vu \cite{CV08} that as
long as $A_n$ is the adjacency matrix of the Erd\H{o}s-R\'enyi graph
$G(n,p)$ with $\frac{(1+\epsilon)\log n}{n} \le p \le 1/2$, we have
$p_n=o(1)$. Here we note that the threshold $\frac{\log n}{n}$ is
optimal. 
For Laplacians of Erd\H{o}s--R\'{e}nyi
random graphs, Clancy, Leake, and Payne \cite{Clancy2015a}, conjectured a limiting distribution for their cokernels, which was proven by the second author \cite{Wood2017}. 

For symmetric matrices with further constraints, a popular
conjecture, first appeared in \cite[Question
10.1]{CV08} by Costello and Vu, and then subsequently in \cite[Conjecture
8.4]{Vu2008}, \cite[Conjecture 5.8]{Vu2014}, and \cite[Section 9,
Problem 7]{Frieze2014}, says that the adjacency matrix of random
regular graphs are non-singular asymptotically almost surely. 

\begin{conjecture}\label{conj:sym} Let $3\le r\le n$. Let $G(n,r)$ be a uniform random
  $r$-regular simple graph on the vertex set $\{1,\dots,n\}$ and let
  $A_{n,r}$ be the adjacency matrix of $G(n,r)$, then $p_n=o(1)$.
\end{conjecture}

Compared to the non-symmetric model, this problem is much less studied. There is recent work by Landon, Sosoe and Yau
\cite{Landon2016} where it can be deduced that $p_n=o(1)$ as long as $r \ge n^\varepsilon$
for any given $\varepsilon>0$. It seems plausible that the method
there can be extended all the way to $r = \omega(\log n)$, but it
seems to the current authors that the singularity problem for smaller
$r$ requires significantly new ideas. 
M\'esz\'aros \cite{Meszaros2018} has recently shown that $p_n\ra 0$ for the adjacency matrices of the random multi-graphs on an even number of vertices that are given by the union of $r$ independent perfect matchings, as well as determining the cokernel distribution of the Laplacians of theses graphs.  In this note we prove the following result on the distribution of cokernels of the adjacency matrices, using results from \cite{Meszaros2018}  and \cite{Wood2017} as discussed above.
\begin{theorem}\label{T:sandpile}
For an integer $r\geq 3$, and an even integer $n,$ let $C_n$ be the adjacency matrix of the multigraph given by taking the union of $r$ independent uniform random perfect matchings on $n$ labelled vertices.  
For  a finite set $P$ of odd primes not dividing $r$, and a finite abelian group $V$ such that $|V|$ is a product of powers of primes in $P$, we have
\begin{equation}\label{E:probs}
\lim_{\substack{n\ra\infty\\ n \textrm{ even}}} \P( \cok C_n\tensor\prod_{p\in P} \Z_p \isom V)=
 \frac{\#\{\textrm{symmetric, bilinear, perfect } \phi:V\times V \ra \C^* \}}{|V||\Aut(V)|}\prod_{p\in P} \prod_{k\geq 0}(1-p^{-2k-1}).
\end{equation}
\end{theorem}
The distribution in Theorem~\ref{T:sandpile} is the same as appears in the  cases of independent entries, Laplacians of Erd\H{o}s--R\'{e}nyi
random graphs \cite{Wood2017}, and in the Laplacian version of the underlying graphs in Theorem~\ref{T:sandpile} \cite{Meszaros2018}.
We also explain in Section~\ref{S:ns} how the limiting distribution of cokernels of these matrices can be used to give a different proof of M\'esz\'aros's result on the asymptotic nonsingularity of the $C_n$, and point out that M\'esz\'aros's result \cite[Proposition 5]{Meszaros2018} has the following simple corollary for the uniform $r$-regular graphs.
\begin{theorem}\label{T:main}
Given a fixed $r\geq 3$, and $p_n$ as in Conjecture~\ref{conj:sym}, we have
$$
\lim_{\substack{n\ra\infty\\ n \textrm{ even}}} p_n=0.
$$
\end{theorem}


\subsection{Notation}
For a prime $p$, we write $\Z_p$ for the $p$-adic integers.
For a finite abelian group $V$ and a prime $p$, we write $V_p$ for $V\tensor \Z_p$.  
If $V$ is a finite abelian group, then $V_p$ is the Sylow $p$-subgroup of $V$.
If $V=\Z^k \times T$, for a finite abelian group $T$, then $V_p=\Z_p^k \times T_p$.
For a set of primes $P$, we let $V_P=\prod_{p\in P}V_p$.

For two groups $G,H$, we write $\Hom(G,H)$ for the set of group homomorphisms from $G$ to $H$ and $\Sur(G,H)$ for the set of surjective group homomorphisms from $G$ to $H$.  We write $\Aut(G)$ for the set of group automorphisms of $G$.  We use $\isom$ to denote an isomorphism of groups.

For an $n\times n$ integer matrix $M$, we write $\cok M$ for $\Z^n/M(\Z^n)$, i.e. the cokernel of the map $M:\Z^n\ra \Z^n$.

We use $\P$ for probability and $\E$ for expectation.
\section{Cokernel distributions for non-symmetric matrices: proof of Theorem~\ref{T:sandpile}}

In this section we give the proof of Theorem~\ref{T:sandpile}. 
A.~M\'esz\'aros \cite{Meszaros2018} has recently studied the sandpile group of the graphs in the theorem, and we first explain a key result of his that we will use.
To start, we let $V$ be any abelian group. We consider vectors $q\in V^n=\Hom(\Z^n, V)$, with coordinates $q_i\in V$ that give the image of the $i$th standard basis vector.
For a vector $q\in V^n$, let $\operatorname{MinCos}_q$ be the minimal coset of $V$ containing all the entries $q_i$ of $q$.  If $\operatorname{MinCos}_q=\gamma +H$ for $\gamma\in V$ and $H$ a subgroup of $V$, we write $r\cdot \operatorname{MinCos}_q$ for $r\gamma +H$.  
For $s\in V^n$, we define $\langle q,s \rangle\in V\tesnor V$ to be $\sum_{i=1}^n q_i\tensor s_i$. By definition of $C_n$, we have that
$\langle q,qC_n \rangle$ is a sum of elements of the form $\langle q,q' \rangle$ where $q'$ is obtained from $q$ by performing an involution with no fixed points on the coordinates.  Thus $\langle q,qC_n \rangle$ is in the subgroup $I_2(V)$ of $V\tensor V$ generated by elements of the form $a\tensor b+b \tensor a$.  
Let $R^S(q,r)=\{s\in (r\cdot \operatorname{MinCos}_q)^n | \langle q,s \rangle \in I_2(V) \textrm{ and } \sum_{i=1}^n s_i = r \sum_{i=1}^n q_i \}$.
Note that $q C_n \in R^S(q,r)$.    
M\'esz\'aros \cite[Theorem 4]{Meszaros2018} proved that 
 $$
\lim_{\substack{n\ra\infty\\ n \textrm{ even}}}  \sum_{\substack{q\in V^n 
}} \max_{s\in R^S(q,r)} |\P(q C_n =s)-|R^S(q,r)|^{-1} |=0.
 $$
M\'esz\'aros used this result to determine the moments of the cokernel of the graph Laplacian of $\hat{G}(n,r)$, and we can similarly use it to determine the moments of the cokernel $\cok C_n$ of the adjacency matrix.
 Note that $q\in \Hom(\Z^n, V)$ is surjective if and only if the coordinates $q_i$ of $q$ generate $V$, and that $q$ descends to a homomorphism from $\cok C_n$ if and only if $q C_n=0$.  So $\E(|\Sur(\cok C_n,V)|)$  is exactly the expected number of $q\in V^n$ such that $q C_n=0$ and the $q_i$ generate $V$ (as in \cite[Section 3]{Wood2017} or \cite[Proposition 34]{Meszaros2018}).  
We conclude
\begin{align*}
\lim_{\substack{n\ra\infty\\ n \textrm{ even}}} \E( |\Sur(\cok C_n,V)|)&= \lim_{\substack{n\ra\infty\\ n \textrm{ even}}} \sum_{\substack{q\in \Sur(\Z^n,V)
\\ 0\in R^S(q,r)
}}
|\P(q C_n=0)|\\
&= \lim_{\substack{n\ra\infty\\ n \textrm{ even}}} \sum_{\substack{q\in \Sur(\Z^n,V)
\\ 0\in R^S(q,r)
}}
|R^S(q,r)|^{-1}. \\
\end{align*}

Now we suppose that $|V|$ is odd, and we will show that for $q\in \Sur(\Z^n,V)$ with $\operatorname{MinCos}_q=V$, we have
$|R^S(q,r)|=|V|^{n-1}/|\wedge^2 V|,$ where $\wedge^2 V$ is the abelian group that is the quotient of $V\tensor V$ by the subgroup generated by elements of the form $a\tensor a$.  Note that since $|V|$ is odd, we have that $I_2(V)$ is the subgroup generated by elements of the form $a\tensor a$, and so $V/I_2(V)=\wedge^2 V$.  Our claim about the size of $|R^S(q,r)|$ will then follow from the fact that the map
\begin{align*}
V^n &\ra \wedge^2 V \times V\\
s &\mapsto \left( \sum_{i=1}^n q_i \wedge s_i , \sum_{i=1}^n s_i \right)
\end{align*}
is surjective.  The surjectivity to the first factor follows, because for $v,w\in V$, if $v=\sum_{i=1}^n a_iq_i$ for $a_i\in \Z$ (which we have from the surjectivity of $q$), then $s$ with $s_i= a_iw$ maps to $v\wedge w$ in the first factor.
Note if $s \mapsto (b,c)$ then if $s'$ has $s'_i=s_i$ for $i\neq j$ and $s_j=s_j+ q_j$, then
$s'\mapsto (b,c+q_j)$, and it follows the map above is surjective.

For the rest of this proof, we assume that $V$ is as in the theorem statement.
We will count the number of $q\in \Sur(\Z^n,V)$
such that $0\in\R^S(q,r)$.  Note that $0\in\R^S(q,r)$ is equivalent to $0\in r\cdot\operatorname{MinCos}_q$ and  $r \sum_{i=1}^n q_i=0$.
Since $r$ is relatively prime to $|V|$, that is equivalent to $\operatorname{MinCos}_q$ being a subgroup and $\sum_{i=1}^n q_i=0$.
Note that $q\in \Sur(\Z^n,V)$ implies that the $q_i$ generate $V$, so any subgroup that all the $q_i$  belong to in fact must be $V$.  
Conversely, if $q\in\Hom(\Z^n,V)$ has $\operatorname{MinCos}_q=V$, then $q\in \Sur(\Z^n,V)$.
So
$$
\#\{ q\in \Sur(\Z^n,V) | 0\in\R^S(q,r) \}\leq  \#\{ q\in \Hom(\Z^n,V) | \sum_{i=1}^n q_i=0 \}=|V|^{n-1}
$$
and
$$
\#\{ q\in \Sur(\Z^n,V) | 0\in\R^S(q,r) \}\geq \#\{ q\in \Hom(\Z^n,V) | \sum_{i=1}^n q_i=0 \} - \sum_{\substack{\gamma+H\\ \textrm{ proper coset of $V$}}}
\#\{ q\in  (\gamma+H)^n \}.
$$
Note that 
$$\sum_{\substack{\gamma+H\\ \textrm{ proper coset of $V$}}} |H|^n=o(|V|^n)$$ (where the constant in the little $o$ notation may depend on $|V|$).
So
$$
\#\{ q\in \Sur(\Z^n,V) | 0\in\R^S(q,r) \}=|V|^{n-1}+o(|V|^n).
$$


 Putting together the above, for any $V$ as in the theorem statement, we conclude that
\begin{align*}
\lim_{\substack{n\ra\infty\\ n \textrm{ even}}} \E( |\Sur(\cok C_n,V)|)&=
  |\wedge^2 V|.
\end{align*}
  Let $H_n$ be a random $n\times n$ symmetric matrix with entries on and above the diagonal independently drawn from Haar measure in $\prod_{p\in P} \Z_p$.
 By \cite[Corollary 9.2]{Wood2017}, we have
 \begin{align*}
\lim_{\substack{n\ra\infty}} \P( \cok H_n\isom V)&=
 \frac{\#\{\textrm{symmetric, bilinear, perfect } \phi:V\times V \ra \C^* \}}{|V||\Aut(V)|}\prod_{p\in P}\prod_{k\geq 0}(1-p^{-2k-1}).
\end{align*} 
By \cite[Theorem 11]{Clancy2015} (or see \cite[Theorem 6.1]{Wood2017}) we have
 \begin{align*}
\lim_{\substack{n\ra\infty}} \E( |\Sur(\cok H_n,V)|)&=|\wedge^2 V|.
\end{align*} 
Thus by \cite[Theorem 8.3]{Wood2017}, which says that these limiting moments determine a unique limiting distribution, we conclude for every integer $a$ that is a product of powers of the primes in $P$, 
 \begin{align*}
\lim_{\substack{n\ra\infty\\ n \textrm{ even}}} \P( \cok C_n \tensor \Z/a\Z \isom V)=\lim_{\substack{n\ra\infty}} \P( \cok H_n\tensor \Z/a\Z \isom V).
\end{align*}
Note that if $W$ is a finitely generated abelian group and $V$ is a finite abelian group  as  in the theorem statement, and $V$ has exponent $a=\prod_{p\in P } p^{e_p}$, then $V\tensor \Z/a'\Z\isom V$, and further, we have $W_P \isom V$ if and only if
$W \tensor \Z/a'\Z \isom V$, where $a'=\prod_{p\in P } p^{e_p+1}$. From this we conclude the theorem.

\section{Directed graph (non-symmetric matrix) analog: proof of Theorem~\ref{T:sandpile2}}\label{S:directed}

In this section, we give a proof of Theorem~\ref{T:sandpile2}, which is  simpler than our proof of Theorem~\ref{T:sandpile} above.
We start with a result of M\'esz\'aros \cite{Meszaros2018}, which he proved in order to determine the asymptotic distribution of the sandpile groups with sink of the directed graphs associated to $D_{n}$.
Let $V$ be any abelian group. We consider vectors $q\in V^n=\Hom(\Z^n, V)$, with coordinates $q_i\in V$ that give the image of the $i$th standard basis vector.
Recall the notation $\operatorname{MinCos}_q$ from the proof of Theorem~\ref{T:main}.
Let $R(q,r)=\{s\in (r\cdot \operatorname{MinCos}_q)^n |  \sum_{i=1}^n s_i = r \sum_{i=1}^n q_i \}$.
Note that $q D_n \in R(q,r)$.    
M\'esz\'aros \cite[Theorem 3]{Meszaros2018} proved that 
\begin{equation}\label{E:M}
\lim_{\substack{n\ra\infty}}  \sum_{\substack{q\in V^n \\ q_i \textrm{ generate $V$}
}} \max_{s\in R(q,r)} |\P(q C_n =s)-|R(q,r)|^{-1} |=0.
\end{equation}

As above in the undirected case and in \cite{Meszaros2018} for the graph Laplacian, we use this to determine the moments of $\cok D_n$.
 As in the proof of Theorem~\ref{T:sandpile}, we  conclude
\begin{align*}
\lim_{\substack{n\ra\infty}} \E( |\Sur(\cok D_n,V)|)&= \lim_{\substack{n\ra\infty}} \sum_{\substack{q\in \Sur(\Z^n,V)
\\ 0\in R(q,r)
}}
|\P(q D_n=0)|\\
&= \lim_{\substack{n\ra\infty}} \sum_{\substack{q\in \Sur(\Z^n,V)
\\ 0\in R(q,r)
}}
|R(q,r)|^{-1}. \\
\end{align*}

For $q\in \Sur(\Z^n,V)$ with $\operatorname{MinCos}_q=V$, we have
$|R(q,r)|=|V|^{n-1}$. 
For the rest of this proof, we assume that $r$ and $|V|$ are relatively prime.
The proof of Theorem~\ref{T:main} shows that for
$q\in \Sur(\Z^n,V)$, and $r$ relatively prime to $|V|$, and $0\in\R^S(q,r)$, we have $\operatorname{MinCos}_q=V,$ and further shows that
$$
\#\{ q\in \Sur(\Z^n,V) | 0\in\R(q,r) \}=|V|^{n-1}+o(|V|^n).
$$


For $V$ as in the theorem statement, we have concluded above that 
\begin{align*}
\lim_{\substack{n\ra\infty}} \E( |\Sur(\cok D_n,V)|)&=
1.
\end{align*}
Thus by \cite[Theorem 3.1 and Proof of Corollary 3.4]{Wood2015a} (see also \cite[Proposition 8.3]{Ellenberg2016} for the case $P=\{p\}$), we conclude the theorem.

\section{Nonsingularity over $\R$}\label{S:ns}
In this section, we show the relationship between limiting cokernel distributions and non-singularity.  The key is that singular matrices can be detected by their infinite cokernels, even after tensoring with $\Z_p$.  

\begin{lemma}\label{L:Fatou}
If $G_n$ for $n\geq 0$ is a sequence of random abelian groups, and for each finite abelian group $W$, the limit
$\lim_{n\ra\infty} \P(G_n\isom W)$ exists, and $\mu$ defined by $\mu(W):= \lim_{n\ra\infty} \P(G_n\isom W)$ for each finite abelian group $W$ is a probability measure on the set finite abelian groups, then $\lim_{n\ra\infty} \P(G_n \textrm{ is infinite})=0$.
\end{lemma}

\begin{proof}
We have
\begin{align*}
\lim_{\substack{n\ra\infty}}  \P(G_n \textrm{ is finite})
&=\lim_{\substack{n\ra\infty}}  \sum_{W \textrm{ fin. ab. group}} \P(G_n\isom W)\\
&\geq \sum_{W \textrm{ fin. ab. group}}\lim_{\substack{n\ra\infty}}  \P(G_n\isom W)\\
&= \sum_{W \textrm{ fin. ab. group}} \mu(W)\\
&=1,
\end{align*}
where the inequality is given by Fatou's Lemma.
\end{proof}

However, it does not suffice to apply Lemma~\ref{L:Fatou} to the cokernels of our sequences of matrices, because for all the sequences $A_n$ of matrices considered in this paper we have $\lim_{n\ra\infty} \P(\cok A_n\isom W)=0$ for each finite abelian group $W$ (since $\prod_{p \textrm{ prime}} (1-p^{-1}))=0$, see the argument for \cite[Corollary 9.3]{Wood2017}).  However, if we tensor our cokernels with $\Z_p$ we can use the following.

\begin{corollary}\label{C:getsing}
If $p$ is a prime, $A_n$ is a sequence of random integral matrices, and $\mu$ a probability measure on finite abelian $p$-groups such that $\lim_{n\ra\infty} \P((\cok A_n)_p\tensor W)=\mu(W)$ for every finite abelian $p$-group, then
$\lim_{n\ra\infty} \P(\det A_n=0)=0$.
\end{corollary}
\begin{proof}
If $M$ is an $n\times n$ integral matrix with $\det M=0$, then $\cok M\isom \Z^k \oplus T$ for some $k\geq 1$ and finite abelian group $T$, and in particular
$(\cok C_n)_P=(\cok C_n)\tensor \Z_p$ is infinite, and so the theorem follows.
\end{proof}

Next we will see how the above results can be used to give proofs of Theorem \ref{T:main} and the directed analog.  We only further need the facts that the limiting distribution is a probability distribution, and arguments of contiguity and conditioning to move between different models of random matrices.  We remark that the proofs of Huang \cite{Huang2018} and M\'esz\'aros \cite{Meszaros2018} only use the $\Z/p\Z$-moments of the cokernels distributions for infinitely many primes $p$ to establish asymptotic nonsingularity.  We have used all of the moments of the cokernel distributions to determine the entire distribution of the cokernels above, but for the nonsingularity argument only need the part that follows from the $W$-moments (i.e. $\E(|\Sur(-,W)|$) for each finite abelian $p$-group $W$ for a single prime $p$.

\begin{proof}[Proof of Theorem~\ref{T:main}]
We first use work of Bollob\'as \cite{Bollobas1980}, to replace $G(n,r)$ with a random multi-graph $G^*(n,r)$ given as follows (see \cite[Corollary 2.18]{Bollobas2001}).  We take $nr$ half-edges
labelled $(i,j)$ for $1\leq i\leq n$ and $1\leq j\leq r$, and then choose a uniform random perfect matching on these $nr$-half edges.  Then we create the random $r$-regular multi-graph $G^*(n,r)$ on vertex set $\{1,\dots,n\}$ by making an edge $ab$ for each pair $(a,j),(b,k)$ in the matching.  If we condition on $G^*(n,r)$ having no loops or multiple edges, then we get exactly $G(n,r)$.  Let $G'(n,r)$ be given by $G^*(n,r)$ conditioned on having no loops.  Since, for fixed $r$, the probability that 
$G^*(n,r)$ has a loop or multiple edge is bounded away from $1$ (see \cite[Proof of Theorem 2.16]{Bollobas2001}), it  suffices to prove the theorem with $G(n,r)$ replaced by $G^*(n,r)$ or $G'(n,r)$.

By work of Janson, $G'(n,r)$ is contiguous with the random multi-graph  $\hat{G}(n,r)$, which is given
by taking the union of $r$ independent uniform random perfect matchings on the vertex set $\{1,\dots,n\}$ (\cite[Theorem 11]{Janson1995}, see also \cite[Theorem 3]{Molloy1997}).  To say that two sequences of random graphs are \emph{contiguous} means that a property that holds asymptotically almost surely for one sequence holds asymptotically almost surely for the other sequence.  Thus is suffices to replace $G'(n,r)$ with $\hat{G}(n,r)$.  Note that the adjacency matrix of $\hat{G}(n,r)$ is the matrix $C_n$ in Theorem~\ref{T:sandpile}.  (M\'esz\'aros \cite[Proposition 5]{Meszaros2018} gives asymptotic nonsingularity of these matrices, and thus Theorem~\ref{T:main}
follows in this way, though we provide a different approach below.)

For a fixed odd prime $p$, not dividing $r$, let $\mu(V)$ denote the right-hand side of Equation~\eqref{E:probs} when $P=\{p\}$.  We will now show that $\mu$ gives a probability measure on finite abelian $p$-groups.
Let $\mathcal{A}$ denote the set of (isomorphism classes of) pairs $(V,\delta)$ where $V$ is a finite abelian $p$-group and $\delta$ is a
symmetric, bilinear, perfect pairing $\delta: V \times V \ra \C^*$.
By \cite[Proposition 7]{Clancy2015}, we have
$$
\sum_{(W,\delta)\in \mathcal{A}} \frac{1}{|W||\Aut(W,\delta)|}=\prod_{k\geq 0}(1-p^{-2k-1})^{-1}.
$$
We have that $\Aut(W)$ acts naturally on the set $P_W$ of symmetric, bilinear, perfect pairings on $V$, with stabilizer of $\delta$ being $\Aut(W,\delta)$,
and its orbits are in bijection with isomorphism classes $(W,\delta)\in \mathcal{A}$ such that $W\isom V$.  
By the orbit-stabilizer theorem, for each finite abelian $p$-group $V$, we have
$$
\sum_{\substack{(W,\delta)\in \mathcal{A}\\ W\isom V}} \frac{1}{|W||\Aut(W,\delta)|}=
\sum_{\delta\in P_W} \frac{1}{|W||\Aut(W)|},
$$
and so, summing over $W$, we have
$$
\prod_{k\geq 0}(1-p^{-2k-1})^{-1}=\sum_{W \textrm{ fin. ab. $p$-group}} \frac{\#\{\textrm{symmetric, bilinear, perfect } \phi:W\times W \ra \C^* \}}{|W||\Aut(W)|}.
$$
Thus, we conclude $\mu$ is a  probability measure on finite abelian $p$-groups, and the theorem follows from Corollary~\ref{C:getsing}.
\end{proof}

In the nonsymmetric matrix case we can conclude the following.
\begin{theorem}\label{T:main2}
Given an integer $r\geq 3$ and an integer $n$, let $D_{n}$ be the sum of $r$ independent uniform random $n\times n$ permutation matrices.
We have
$$
\lim_{\substack{n\ra\infty}} \P(\det D_n=0)=0.
$$
The theorem is also true if we replace $D_n$ by the adjacency matrix of a uniform random $r$-regular (in-degree $r$ and out-degree $r$) random directed graph on $n$ labeled vertices.
\end{theorem}

A version of Theorem~\ref{T:main2} for the configuration model of random regular directed graphs was  proven by Huang \cite[Theorem 1.1]{Huang2018} (and other models including Theorem~\ref{T:main2} can follow from his work by the same types of contiguity and conditioning arguments that allow us to move between models). 

\begin{proof}[Proof of Theorem~\ref{T:main2}]
For a finite abelian $p$-group $V$, let $\mu'(V)$ denote the right-hand side of Equation~\eqref{E:probs2} when $P=\{p\}$. 
The measure $\mu'$ is  the well known Cohen-Lenstra probability measure on finite abelian $p$-groups (see e.g. \cite{Hall1938} for a proof that is is a probability measure).
So the first statement of the theorem follows from Corollary~\ref{C:getsing}.

Let $\Gamma(n,r)$ be a uniform random $r$-regular (in-degree $r$ and out-degree $r$) random directed graph on $n$ labeled vertices.
Let $\bar{\Gamma}(n,r)$ be the random multi-graph given by a union of $r$ independent uniform 1-regular (no loops or multiple edges) directed graphs on $n$ labeled vertices.
Let  $\tilde{\Gamma}(n,r)$ be $\bar{\Gamma}(n,r)$ conditioned on no multiple edges.
By \cite[Section 4]{Molloy1997}, we have that $\Gamma(n,r)$ is contiguous with $\tilde{\Gamma}(n,r)$ (see also \cite[Section 4]{Janson1995}), so it suffices to prove the theorem for adjacency matrices of $\tilde{\Gamma}(n,r)$.
  Since the probability that $\bar{\Gamma}(n,r)$ has multiple edges is bounded away from  $1$ (e.g. see \cite[Theorem 7]{Janson1995}),
  it suffices to prove the theorem for adjacency matrices of $\bar{\Gamma}(n,r)$.
  If we let $\Gamma^*(n,r)$ be the random graph whose adjacency matrix is the sum of $r$ independent uniform random $n\times n$ permutation matrices,
  then $\bar{\Gamma}(n,r)$ is $\Gamma^*(n,r)$ conditioned on no loops.  Since the probability that $\Gamma^*(n,r)$ has loops is bounded away from $1$ (e.g. see \cite[Theorem 7]{Janson1995}),
  it suffices to prove the theorem for adjacency matrices of $\Gamma^*(n,r)$, which is what we have already done.
\end{proof}

\section{New distributions}\label{S:newd}
The matrices $D_n$ and $C_n$ from Theorems~\ref{T:sandpile2} and \ref{T:sandpile} have each row and column summing to $r$. 
So when $p$ is a prime such that $p\mid r$, these matrices are always singular mod $p$, and $(\cok D_n)_p$ and $(\cok C_n)_p$ are never trivial, and in particular these cokernels  are not distributed in the familiar distributions given in Theorems~\ref{T:sandpile2} and \ref{T:sandpile}.  In this section we give some remarks on the distribution in the non-symmetric matrix case, and find the limiting moments of this distribution.

For integers $r,m$, we define an \emph{$(r,m)$-pair} to be a pair $(G,C)$, where $G$ is an abelian group $G$, and $C$ is a coset of $G$, such that the following conditions hold: if we write $C=\gamma+H$, where $\gamma\in G$ and $H$ is a subgroup of $G$, then (1) $r(G/H)=0$, (2) $\gamma$ generates $G/H$, and (3) $mr\gamma\in rH$.  Note that if $G$ is finite and $r$ is relatively prime to $|G|$, then condition (1) implies $H=G$ and thus $C=G$.  

We note that $\cok D_n=\Z^n/D_n(\Z^n)$ is not just an abelian group, but naturally has the structure of an $(r,n)$-pair with coset $e_1+E$, where $e_i$ are the images of the standard generators of $\Z^n$, and $E$ is the subgroup generated by the elements $e_i-e_j$ for $1\leq i <j \leq n$.  Condition (2) would be satisfied for any matrix.  If we let $d_{i,j}$ be the entries of $D_n$, then for each $j$, we have $\sum_i d_{i,j}=r$.  Since in $\cok D_n$, we have $\sum_i d_{i,j} e_i=0$, in $(\cok D_n)/E$, we have $re_1=0$,  and thus condition (1) is satisfied.  Since $\sum_{i,j} d_{i,j} e_i=0$, we have $0=nre_1+\sum_{i,j} d_{i,j} (e_i-e_1)=nre_1+\sum_{i} r(e_i-e_1)$, which gives condition (3).

For abelian groups $G,G'$ with cosets $C,C'$ respectively, we let $\Sur((G,C),(G',C'))$ be the set of surjective group homomorphisms from $G$ to $G'$ in which $C$ has image $C'$.  For a random $(r,m)$-pair $(G,C)$, we define its moments to be indexed by fixed $(r,m)$-pairs $(A,B)$, and for the $(A,B)$-moment to be $\E(|\Sur((G,C),(A,B))|).$
For a group $G$, we let $G[r]$ denote the subgroup of elements $g\in G$ such that $rg=0$.
 
\begin{theorem}\label{T:sandpile3}
Let $r$ and $m$ be positive integers. Let $D_n$ be as in Theorem~\ref{T:sandpile2}, and $e_1$ and $E$ as above.  Let $(V,B)$ be an $(r,m)$-pair such that $B$ is a coset of the subgroup $H$ of $V$.
Then 
$$
\lim_{\substack{n\ra\infty\\ m|n}} \E(|\Sur((\cok D_n,e_1+E),(V,B))|)= |H[r]|.
$$
\end{theorem}

\begin{proof}
Let $\tilde{e}_i$ be the standard generators of $\Z^n$ and $\tilde{E}$ the subgroup generated by $\tilde{E}$.  
Note that a surjection from $(\cok D_n,e_1+E)$ to $(V,B)$ is exactly given by surjection $q:(\Z^n,\tilde{e}_1+\tilde{E})\ra (V,B)$  (we write $q\in V^n$ as in the proof of Theorem~\ref{T:sandpile2}) such that $q D_n=0$.
A surjection $q:(\Z^n,\tilde{e}_1+\tilde{E})\ra (V,B)$ is exactly a surjection $q: \Z^n \ra V$ such that $\operatorname{MinCos}_q=B$.  (Recall the notation from the proof of Theorem~\ref{T:sandpile}, and note that the minimal coset containing elements $q_i$ is exactly $q_1+H$, where $H$ is the subgroup generated by all the $q_i-q_j$.) 

We write $B=\gamma+H$, for some $\gamma\in V$.
For $q\in \Sur(\Z^n, V)$ with $\operatorname{MinCos}_q=B$, we next determine $|R(q,r)|$. 
(Recall the notation from the proof of Theorem~\ref{T:sandpile2}.)
 By condition (1), we have $r\cdot B=H$.  We have that $R(q,r)$ is the set of $s\in H^n$ such that $\sum_{i=1}^n s_i = r\sum_{i=1}^n q_i$.  By condition (1), we have $rq_i\in H$, and so   $|R(q,r)|=|H|^{n-1}$.

Next we determine the number of $q\in \Sur(\Z^n, V)$ with $\operatorname{MinCos}_q=B$ such that $0\in R(q,r)$.  Note that $0\in R(q,r)$ is equivalent to $0\in r\cdot \operatorname{MinCos}_q$ and $r\sum_{i=1}^n q_i=0$.  For $q$ with $\operatorname{MinCos}_q=B$, we have that $0\in r\cdot \operatorname{MinCos}_q$ is implied by condition (1) on $B$.  If we let $q_1,\dots, q_{n-1}$ be any choices of elements in $B$, then we will count the number of  $q_n\in B$ such that 
$r\sum_{i=1}^n q_i=0$.  This is the same as the number of $h$ such that $r(\sum_{i=1}^{n-1} q_i +\gamma + h)=0$, which is the same as the number of $h'$ such that
$r(n\gamma +h')=0$.  By condition (3) and $m\mid n$, we have $rn\gamma\in rH$, and thus there are $H[r]$ choices of $q_n$ satisfying $r\sum_{i=1}^n q_i=0$.
We conclude there are $|H|^{n-1}|H[r]|$ choices of $q\in B^n$ such that $r\sum_{i=1}^n q_i=0$.
  By condition (2), we have $q\in V^n$ is a surjection if and only if the $q_i-q_j$ generate $H$.  Also, for $q\in B^n$, we have $\operatorname{MinCos}_q=B$ if and only if the $q_i-q_j$ generate $H$.
If we choose $q_1$ and then $q_i-q_i$ for each $i$, we see that the number of $q\in B^n$ such that the $q_i-q_j$ do not generate $H$ is a most
$$
|H| \sum_{\substack{H' \textrm{proper}\\ \textrm{subgroup of $H$}}} |S|^{n-1}=o(|H|^n),
$$
where the constant in the little $o$ notation depends on $H$.  So we conclude that the number of $q\in \Sur(\Z^n, V)$ with $\operatorname{MinCos}_q=B$ such that $0\in R(q,r)$ is $|H|^{n-1}|H[r]|+o(|H|^n)$.

Thus, using Equation~\eqref{E:M}, we conclude
\begin{align*}
\E(|\Sur((\cok D_n,e_1+E),(V,B))|)&=  \sum_{\substack{q\in \Sur(\Z^n,V)\\\operatorname{MinCos}_q=B
\\ 0\in R(q,r)
}}
|\P(q D_n=0)|\\
&=  \sum_{\substack{q\in \Sur(\Z^n,V)\\\operatorname{MinCos}_q=B
\\ 0\in R(q,r)
}}
|R(q,r)|^{-1} \\
&= |H[r]|+o(1).
\end{align*}
\end{proof}

It is natural to combine these moments into the usual moments of finite abelian groups.
For a finite set of primes $P$ and an integer $n$, we define $n_P$ to be the largest divisor of $n$ that is a product of powers of primes in $P$.

\begin{corollary}\label{C:mom}
Let $r$ and $m$ be positive integers. Let $D_n$ be as in Theorem~\ref{T:sandpile2}.  Let $P$ be a finite set of primes and $V$ be a finite abelian group whose order is a product of powers of primes in $P$.  
Then 
$$
\lim_{\substack{n\ra\infty\\ n_P=m}} \E(|\Sur(\cok D_n,V)|) =\sum_{\substack{\gamma+H \textrm{coset of $V$}\\
(V,\gamma+H) \textrm{ an $(r,m)$-pair}
}} |H[r]|.
$$
\end{corollary}

\begin{proof}
It suffices to show that every surjection $\phi:\cok D_n \ra V$ has $(V,\phi(e_1+E))$ an $(r,m)$-pair, so each surjection is a surjection to some $(r,m)$-pair.  Since $r((\cok D_n)/E)=0$, we have
$r(V/\phi(E))=0$.  Since $e_1$ generates $(\cok D_n)/E$, we have that $\phi(e_1)$ generates $V/E$.  Since $nre_1 \in rE$, we have
$nr\phi(e_1) \in r\phi(E)$, and since $n_P=m$ that implies $mr\phi(e_1) \in r\phi(E).$
\end{proof}

We notice two interesting features of the moments in Theorem~\ref{T:sandpile3}.
If $r$ is relatively prime to the $p\in P$, then we noted above that $(V,V)$ is the only $(r,m)$-pair, and the moments are all $1$.
The first interesting feature, as expected above, is that the moments will be bigger when $r$ and $V$ are not relatively prime.  For example, if $V=\Z/p\Z$, and $p\mid r$,
then each non-zero element $\gamma\in V$ gives an $(r,m)$-pair with trivial $H$, and we also have the $(r,m)$-pair $(V,V)$ with $|V[r]|=p$, and so we have total limiting $\Z/p\Z$-moment $2p-1$.

We also note that for $r$ and $V$ fixed, the moments also can depend on $n_P$, i.e. there is not necessarily a limit as $n\ra\infty$ if we include all $n$.  This is why we sort $n$ by the values of $n_P$ in Corollary~\ref{C:mom} above.  For example, let $r=p$ and $V=\Z/p^2\Z$.  The $(r,p)$-pairs have cosets $B=V$ and $B=\gamma+p\Z/p^2\Z$ for $\gamma\not\in p\Z/p^2\Z$ (there are $p-1$ choices of the latter cosets).  However, the only $(r,1)$-pair is with coset $B=V$, because the cosets $B=\gamma+p\Z/p^2\Z$ do not satisfy condition (3), as $p\gamma$ is not trivial but $p (p\Z/p^2\Z)$ is trivial. So for $n$ with $n_p=1$, the $\Z/p^2\Z$-moment approaches $p$, but for
$n$ such that $p\mid n$, the $\Z/p^2\Z$-moment approaches $p^2$. 

In conclusion, for each $r$ and $m$ we see new distributions on finite abelian groups, and it would be interesting to better understand these distributions.  

\subsection*{Acknowledgements} 
The authors thank Nick Cook for helpful comments on an earlier version of this manuscript.
During this work, the first author was partially supported by National
Science Foundation grants DMS-1600782 and DMS-1752345.
During this work, the second author was supported by a Packard Fellowship for Science and Engineering, a Sloan Research Fellowship,  National Science Foundation grants DMS-1652116 and DMS-1301690, and a Vilas Early Career Investigator Award.

\newcommand{\etalchar}[1]{$^{#1}$}

\end{document}